\PassOptionsToPackage{unicode}{hyperref}
\PassOptionsToPackage{hyphens}{url}
\documentclass[
]{article}
\usepackage{lmodern}
\usepackage{amssymb,amsmath,amsthm,appendix}
\usepackage{ifxetex,ifluatex}
\ifnum 0\ifxetex 1\fi\ifluatex 1\fi=0 
  \usepackage[T1]{fontenc}
  \usepackage[utf8]{inputenc}
  \usepackage{textcomp} 
\else 
  \usepackage{unicode-math}
  \defaultfontfeatures{Scale=MatchLowercase}
  \defaultfontfeatures[\rmfamily]{Ligatures=TeX,Scale=1}
\fi
\IfFileExists{upquote.sty}{\usepackage{upquote}}{}
\IfFileExists{microtype.sty}{
  \usepackage[]{microtype}
  \UseMicrotypeSet[protrusion]{basicmath} 
}{}
\makeatletter
\@ifundefined{KOMAClassName}{
  \IfFileExists{parskip.sty}{%
    \usepackage{parskip}
  }{
    \setlength{\parindent}{0pt}
    \setlength{\parskip}{6pt plus 2pt minus 1pt}}
}{
  \KOMAoptions{parskip=half}}
\makeatother
\usepackage{xcolor}
\IfFileExists{xurl.sty}{\usepackage{xurl}}{} 
\IfFileExists{bookmark.sty}{\usepackage{bookmark}}{\usepackage{hyperref}}
\hypersetup{
  hidelinks,
  pdfcreator={LaTeX via pandoc}}
\theoremstyle{theorem}
\newtheorem*{theorem}{Theorem}

\theoremstyle{lemma}
\newtheorem{lemma}{Lemma}

\theoremstyle{proposition}

\theoremstyle{corollary}
\newtheorem*{corollary}{Corollary}

\theoremstyle{remark}

\title{Generalisations of Singular Value Decomposition to dual-numbered matrices}
\author{Ran Gutin \\ Department of Computer Science \\ Imperial College London}
\date{\today}

\begin{document}

\maketitle

\begin{abstract}
  We present two generalisations of Singular Value Decomposition from real-numbered matrices to dual-numbered matrices. We prove that every dual-numbered matrix has both types of SVD. Both of our generalisations are motivated by applications, either to geometry or to mechanics.

\end{abstract}

\hypertarget{introduction}{%
\section{Introduction}\label{introduction}}

In this paper, we consider two possible generalisations of Singular Value Decomposition (\emph{$T$-SVD} and \emph{$*$-SVD}) to matrices over the ring of dual numbers. We prove that both generalisations always exist. Both types of SVD are motivated by applications.

A \emph{dual number} is a number of the form $a+b\epsilon$ where $a, b \in \mathbb R$ and $\epsilon^2 = 0$. The dual numbers form a commutative, assocative and unital algebra over the real numbers. Let $\mathbb D$ denote the dual numbers, and $M_n(\mathbb D)$ denote the ring of $n \times n$ dual-numbered matrices.

Dual numbers have applications in automatic differentiation \cite{10.1007/s11075-015-0067-6}, mechanics (via screw theory, see \cite{Fischer1998DualNumberMI}), computer graphics (via the \emph{dual quaternion} algebra \cite{Kenwright_abeginners}), and geometry \cite{1968}.
Our paper is motivated by the applications in geometry described in Section \ref{sec:Ya} and in mechanics given in Section \ref{sec:mech}.

Our main results are the existence of the \emph{$T$-SVD} and the existence of the \emph{$*$-SVD}. $T$-SVD is a generalisation of Singular Value Decomposition that resembles that over real numbers, while $*$-SVD is a generalisation of SVD that resembles that over complex numbers. (Note that in earlier editions of this paper, we called the $T$-SVD the R-SVD, and the $*$-SVD the C-SVD). We prove the following:

\begin{theorem}[Dual $T$-SVD]
  Given a square dual number matrix \(M \in M_n(\mathbb D)\), we can decompose the matrix as: \[M = U\Sigma V^T\] where \(U^T U = V^T V = I\), and \(\Sigma\) is a diagonal matrix.
\end{theorem}

\begin{theorem}[Dual $*$-SVD]
  Every square dual-numbered matrix \(M \in M_n(\mathbb D)\) can be decomposed as \[M = U\Sigma \overline V^T\] where \(U\overline U^T = V\overline V^T = I\), and \(\Sigma\) is a block-diagonal matrix where each block is of one of the forms \(\begin{pmatrix}\sigma_i\end{pmatrix}\), \(\begin{pmatrix}\sigma_i & -\epsilon \sigma'_i \\ \epsilon\sigma_i' & \sigma_i\end{pmatrix}\) or \(\begin{pmatrix} \epsilon\sigma_i'\end{pmatrix}\), where each $\sigma_i$ and $\sigma_i'$ is real, and $\sigma_i' \neq 0$.
\end{theorem}

\subsection{Dual $*$-SVD}\label{sec:Ya}

Our study of $*$-SVD is motivated by Yaglom's 1968 book \emph{Complex numbers in geometry} (\cite{1968}). Yaglom considers the group of Laguerre transformations, which is analogous to the group of Moebius transformations over the complex numbers.
The Laguerre transformations are the group of functions of the form \(z \mapsto \frac{az + b}{cz + d}\) where \(a, b, c, d\) are elements of \(\mathbb D\), $z$ is a variable over $\mathbb D$, and \(ad-bc\) is not a zero divisor. It can easily be seen that every Laguerre transformation \(z \mapsto \frac{az + b}{cz + d}\) can be represented as the $2 \times 2$ matrix \(\begin{pmatrix}a & b \\ c & d \end{pmatrix}\).

Yaglom classifies the elements of this group in a geometric way. We restate the classification in the language of matrices. Yaglom argues that every invertible $2 \times 2$ matrix over the dual numbers can be expressed in exactly one of the following two ways:

\begin{itemize}
\item
  \(U \Sigma \overline V^T\) where \(U\overline U^T = V\overline V^T = I\) and \(\Sigma\) is a diagonal matrix with real-valued entries.
\item
  \(U \Sigma\) where \(U\overline U^T = I\), \(\Sigma = \begin{pmatrix} \sigma & -\epsilon\sigma' \\ \epsilon \sigma' & \sigma\end{pmatrix}\), and both $\sigma$ and $\sigma'$ are real (and $\sigma'$ is non-zero).
\end{itemize}

The first of these forms resembles Singular Value Decomposition. What we propose in this paper is a generalisation of Singular Value Decomposition to square dual-numbered matrices which includes both forms as special cases.

\subsection{Dual $T$-SVD}\label{sec:mech}

Our study of $T$-SVD is motivated by recent research  \cite{dualmatrix,HAN2018352,Han2018SpectralCM,10.1115/1.4040882} in mechanics, where the authors consider either a form of SVD that is essentially $T$-SVD, or a form of Polar Decomposition that is essentially $T$-Polar
Decomposition. For completeness, we describe $T$-Polar Decomposition below.

\begin{corollary}[Dual $T$-Polar Decomposition]
  Every square dual-numbered matrix $M$ can be expressed in the form $M = UP$ where $U^T U = I$ and $P$ is symmetric.
\end{corollary}

One of the main applications of the $T$-SVD is in finding the Moore-Penrose generalised inverse of a dual matrix (whenever it exists). This has applications in kinematic synthesis (see \cite{10.1115/1.4040882}).

Very recently, Udwadia et al. \cite{pinv} studied the existence of the Moore-Penrose generalised inverse for dual-numbered matrices and showed that unlike in the case of real and complex matrices, not all dual-numbered matrices have such inverses. In contrast to their result, we prove that the $T$-SVD does exist for all dual-numbered matrices. Our result is new as none of the authors above have proved that the $T$-SVD exists in general.

\hypertarget{structure}{%
\subsection{Structure of the paper}\label{structure}}

Section \ref{preliminaries} contains preliminaries.
In section \ref{dual-spectral}, we will prove the $*$/$T$-spectral theorems. In section \ref{dual-svd}, we will show that every square dual-number matrix has a $*$/$T$-SVD.

\section{Preliminaries} \label{preliminaries}

\subsection{Dual numbers}

The dual numbers are the ring $\mathbb R[\epsilon]/(\epsilon^2)$. In other words the dual numbers are pairs of real numbers, usually written as $a + b\epsilon$, with the following operations defined on them:

\begin{itemize}
  \item $(a + b\epsilon) + (c + d\epsilon) = (a + c) + (b + d)\epsilon$.
  \item $(a + b\epsilon)(c + d\epsilon) = ac + (ad + bc)\epsilon$.
  \item $(a + b\epsilon)^{-1} = a^{-1} - \epsilon b a^{-2}$.
\end{itemize}

Given a polynomial with real coefficients $F(X) \in \mathbb R[X]$, evaluating $F$ on a dual number gives $F(a + b\epsilon) = F(a) + \epsilon b F'(a)$ where $F'$ denotes the derivative of $F$. Multiples of the dual number $\epsilon$ are sometimes referred to as ``infinitesimal'' with the intuition being that $\epsilon$ is ``so small'' that it squares to $0$.

The dual numbers contrast with the complex numbers. The complex numbers are defined as the ring $\mathbb R[i]/(i^2 + 1)$. Like the dual numbers, the complex numbers are pairs of real numbers with certain operations defined on them. The difference is that while the complex numbers are defined by adjoining an element $i$ such that $i^2 = -1$, the dual numbers are defined by adjoining an element $\epsilon$ such that $\epsilon^2 = 0$.

Unlike the complex numbers, the dual numbers do not form a field. As such, instead of talking about vector spaces over dual numbers, one must talk about \emph{modules} over dual numbers. Modules over dual numbers don't necessarily have bases: For instance, the module $\epsilon \mathbb D$ (that is, multiples of the dual number $\epsilon$) doesn't have a basis. For this reason, one cannot in general talk about the ``rank'' of a matrix over the dual numbers. A module over the dual numbers (or generally any ring) that does have a basis is referred to as a \emph{free} module.

\subsection{Spectral theorem over symmetric real matrices}\label{ove$T$-symmetri$*$-real-matrices}

The spectral theorem over symmetric matrices states that a symmetric matrix \(A\) over the real numbers can be orthogonally diagonalised. In other words, if \(A\) is a real matrix such that \(A = A^T\), then there exists a matrix \(P\) such that \(P^TP=I\) and \(A = PDP^T\) for some diagonal matrix \(D\).

\subsection{Spectral theorem over skew-symmetric real matrices}\label{ove$T$-skew-symmetri$*$-real-matrices}

The spectral theorem over skew-symmetric matrices states that a skew-symmetric matrix \(A\) over the real number can be orthogonally block-diagonalized, where every block is a skew-symmetric $2 \times 2$ block except possibly for one block, which has dimensions $1\times 1$ and whose only entry equals \(0\).

\hypertarget{a-note-on-notation}{%
\subsection{Notation and terms}\label{a-note-on-notation}}

By \(\overline{a + b\epsilon}\), we mean \(a - b\epsilon\). Given a dual number \(a + b\epsilon\), we call \(a\) the \emph{standard part} and $b$ the \emph{infinitesimal} part. We sometimes denote the standard and infinitesimal parts of a dual number \(z\) by \(\operatorname{st}(z)\) and \(\Im(z)\) respectively. We also use the term \emph{infinitesimal} to describe a dual number whose standard part is zero, and \emph{appreciable} to describe a dual number whose standard part is non-zero. We sometimes call a dual number whose infinitesimal part is zero a real number. All the above terms generalise to matrices and vectors over dual numbers in the obvious way. Likewise, the operations \emph{infinitesimal part} and \emph{standard part} generalise to sets of dual-numbered vectors by applying these operations to each element of the set.

We sometimes write \(M^*\) for \(\overline M^T\).

We use $(u,v)$ to denote $u^T v$, and $\langle u, v \rangle$ to denote $u^T\overline v$. Two vectors are considered \emph{$T$-orthogonal} if $(u,v) = 0$, and \emph{$*$-orthogonal} if $\langle u, v \rangle = 0$.

A dual matrix \(U\) which satisfies \(U\bar U^T = \bar U^T U = I\) is called \emph{unitary}. Similarly, one which satisfies $U^T U = U U^T = I$ is called \emph{$T$-orthogonal}.

A dual matrix \(A\) such that \(A = \bar A^T\) is called \emph{Hermitian}. Similarly, one which satisfies $A = A^T$ is called \emph{symmetric}.

A vector $v$ is called an \emph{eigenvector} of a dual-numbered matrix $A$ if $v$ is appreciable and $Av = \lambda v$ for some $\lambda \in \mathbb D$.

\section{Dual-number $*$/$T$-spectral theorem} \label{dual-spectral}

\begin{theorem}[Dual $*$-spectral]
  Given a Hermitian dual number matrix \(M \in M_n(\mathbb D)\), we can decompose the matrix as: \[M = V\Sigma V^*\] where \(V\) is unitary, and \(\Sigma\) is a block-diagonal matrix where each block is either of the form:
  \begin{itemize}
  \item
    \(\begin{pmatrix}\sigma_i\end{pmatrix}\),
  \item
    or \(\begin{pmatrix}\sigma_i & -\epsilon \sigma'_i \\ \epsilon\sigma_i' & \sigma_i\end{pmatrix}\)
  \end{itemize}
  where each \(\sigma_i\) and \(\sigma_i'\) is real, and $\sigma'_i \neq 0$.
\end{theorem}
\begin{proof}
  Let $M$ be a Hermitian matrix. We find a real matrix $S$ such that $S\operatorname{st}(A)S^T$ is diagonal. We let $M'=SMS^T$, which we write as a block matrix
  $$\begin{pmatrix}
    \lambda_1 I + \epsilon B_{11} & \epsilon B_{12} & \dotsb & \epsilon B_{1n} \\
    -\epsilon B^T_{12} & \lambda_2 I + \epsilon B_{22} & \ddots & \vdots \\
    \vdots & \ddots & \ddots & \epsilon B_{n-1,n} \\
    -\epsilon B^T_{1n} & \dotsb & -\epsilon B_{n-1,n}^T & \lambda_n I + \epsilon B_{nn}
  \end{pmatrix}.$$
  where each $B_{ii}$ is skew-symmetric. We let $$P = \begin{pmatrix}
    I & \frac{\epsilon B_{12}}{\lambda_1 - \lambda_2} & \dotsb & \frac{\epsilon B_{1n}}{\lambda_1 - \lambda_n} \\
    \frac{\epsilon B_{12}^T}{\lambda_1 - \lambda_2} & I & \ddots & \vdots \\
    \vdots & \ddots & \ddots & \frac{\epsilon B_{n-1,n}}{\lambda_{n-1} - \lambda_n}\\
    \frac{\epsilon B_{1n}^T}{\lambda_1 - \lambda_n} & \dotsb & \frac{\epsilon B_{n-1,n}^T}{\lambda_{n-1} - \lambda_n} & I
  \end{pmatrix},$$
  and let $M'' = P M' P^*$. We end up with $M''$ being equal to a direct sum of matrices: $M'' = (\lambda_1 I + \epsilon B_{11}) \oplus (\lambda_2 I + \epsilon B_{22}) \oplus \dotsb \oplus (\lambda_n I + \epsilon B_{nn})$. We finally use the spectral theorem for skew-symmetric matrices to find matrices $Q_i$ such that $Q_i B_{ii} Q_i^T$ is equal to a direct sum of skew-symmetric $2\times2$ blocks except for potentially one zero block. We thus get that $(Q_1 \oplus Q_2 \oplus \dotsb \oplus Q_n) M'' (Q_1 \oplus Q_2 \oplus \dotsb \oplus Q_n)^T$ is a block-diagonal matrix in the desired form.
\end{proof}

\begin{theorem}[Dual $T$-spectral]
  Every symmetric matrix can be orthogonally diagonalised.
\end{theorem}

\begin{proof}
  The proof is the same as for the $*$-spectral case except that
  $$M' = \begin{pmatrix}
    \lambda_1 I + \epsilon B_{11} & \epsilon B_{12} & \dotsb & \epsilon B_{1n} \\
    \epsilon B^T_{12} & \lambda_2 I + \epsilon B_{22} & \ddots & \vdots \\
    \vdots & \ddots & \ddots & \epsilon B_{n-1,n} \\
    \epsilon B^T_{1n} & \dotsb & \epsilon B_{n-1,n}^T & \lambda_n I + \epsilon B_{nn}
  \end{pmatrix},$$
  and 
  $$P = \begin{pmatrix}
    I & \frac{\epsilon B_{12}}{\lambda_1 - \lambda_2} & \dotsb & \frac{\epsilon B_{1n}}{\lambda_1 - \lambda_n} \\
    -\frac{\epsilon B_{12}^T}{\lambda_1 - \lambda_2} & I & \ddots & \vdots \\
    \vdots & \ddots & \ddots & \frac{\epsilon B_{n-1,n}}{\lambda_{n-1} - \lambda_n}\\
    -\frac{\epsilon B_{1n}^T}{\lambda_1 - \lambda_n} & \dotsb & -\frac{\epsilon B_{n-1,n}^T}{\lambda_{n-1} - \lambda_n} & I
  \end{pmatrix},
    $$
and $M'' = PMP^T$. And instead of using the spectral theorem for skew-symmetric matrices, we use the one for symmetric matrices.
\end{proof}

The following theorem can be used to show that the $*$/$T$-spectral decompositions are unique. Note that in order to show that the $*$-spectral decomposition is unique, one must first fully diagonalise the matrix $\begin{pmatrix}\sigma & -\epsilon \sigma' \\ \epsilon \sigma' & \sigma \end{pmatrix}$ to $\begin{pmatrix} \sigma + i\epsilon\sigma' & 0 \\ 0 & \sigma - i\epsilon\sigma'\end{pmatrix}$ by working over $\mathbb C \otimes \mathbb D$.

\begin{theorem}[Uniqueness of eigenvalues]
  An eigenbasis of a linear endomorphism $T : \mathbb D^n \to \mathbb D^n$ corresponds to a unique multi-set of eigenvalues.
\end{theorem}
\begin{proof}
  Let $e_1, e_2, \dotsc, e_n$ be a basis made up of eigenvectors of $T$. Let the corresponding eigenvalues be $\lambda_1, \lambda_2, \dotsc, \lambda_n$ respectively.

  We intend to show that for any eigenvalue $\lambda$ of $T$, the number of ``$e_i$''s in our basis with eigenvalue equal to $\lambda$ is equal to $\dim(\operatorname{st}(V_\lambda))$. Since the quantity $\dim(\operatorname{st}(V_\lambda))$ depends only on $T$, it implies that every eigenbasis has the same amount of eigenvectors with eigenvalue equal to $\lambda$. 

  Assume that $e_1, e_2, \dotsc, e_k$ all have some eigenvalue $\lambda$, and no other $e_i$ has eigenvalue $\lambda$. We will show that $k = \dim(\operatorname{st}(V_\lambda))$ by showing that $\operatorname{st}(e_1), \operatorname{st}(e_2), \dotsc, \operatorname{st}(e_k)$ form a basis of $\operatorname{st}(V_\lambda)$.
  
  \paragraph{Claim 1.} $\operatorname{st}(e_1), \operatorname{st}(e_2), \dotsc, \operatorname{st}(e_k)$ form a spanning set of $\operatorname{st}(V_\lambda)$ (the standard part of the eigenspace corresponding to $\lambda$).
  
  \emph{Proof of claim 1.} Given a vector in $\operatorname{st}(V_\lambda)$, we wish to express it as a linear combination of the above vectors. Our vector is equal to $\operatorname{st}(e)$ for some $e \in V_\lambda$. We can express $e$ as $e = \sum_{i=1}^n \alpha_i e_i$. From this we can conclude that $\operatorname{st}(e) = \sum_{i=1}^n \operatorname{st}(\alpha_i) \operatorname{st}(e_i)$. But notice that this sum includes vectors $\operatorname{st}(e_{k+1})$ to $\operatorname{st}(e_n)$, which we don't want in our linear combination. We thus need to show that $\operatorname{st}(\alpha_i) = 0$ for all $i > k$. To this end, notice that $T(e) = \lambda e = \lambda \left(\sum_{i=1}^n \alpha_i e_i\right)$ and $T(e) = \sum_{i=1}^n \alpha_i T(e_i) = \sum_{i=1}^n \alpha_i \lambda_i e_i$. Since we know that the $e_i$ are linearly independent, we have that $\alpha_i(\lambda_i - \lambda) = 0$ for all $i$. From $\alpha_i(\lambda_i - \lambda) = 0$, one can check that if $\lambda \neq \lambda_i$ then $\alpha_i$ is infinitesimal. Since for all $i > k$, we have that $\lambda_i \neq \lambda$, we must have that $\alpha_i$ is infinitesimal for $i > k$. Recalling that $e = \sum_{i=1}^k \alpha_i e_i + \sum_{i=k+1}^n \alpha_i e_i$, taking standard parts of both sides gives $\operatorname{st}(e) = \sum_{i=1}^k \operatorname{st}(\alpha_i) \operatorname{st}(e_i) + 0$, which is what we sought to show.

  \paragraph{Claim 2.} $\operatorname{st}(e_1), \operatorname{st}(e_2), \dotsc, \operatorname{st}(e_k)$ forms a linearly independent set.
  
  \emph{Proof of claim 2.} Assume that there exist $(\alpha_i)$ (where each $\alpha_i$ is real) such that $\sum_{i=1}^k \alpha_i \operatorname{st}(e_i) = 0$. We then have that $\operatorname{st}\left(\sum_{i=1}^k \alpha_i e_i \right) = 0$. If $\alpha_j \neq 0$ for some $j$ then we get that $\sum_{i=1}^k \epsilon \alpha_i e_i = 0$, which contradicts the linear independence of $e_1, e_2, \dotsc, e_n$. Therefore all $\alpha_i$ equal $0$.
\end{proof}

\section{Dual-number $*$/$T$-Singular Value Decomposition} \label{dual-svd}

These are the theorems we will prove in this section.

\begin{theorem}[Dual $*$-SVD]
  Given a square dual number matrix \(M \in M_n(\mathbb D)\), we can decompose the matrix as: \[M = U\Sigma V^*\] where \(U\) and \(V\) are unitary, and \(\Sigma\) is a block-diagonal matrix where each block is either of the form \(\begin{pmatrix}\sigma_i\end{pmatrix}\), \(\begin{pmatrix}\sigma_i & -\epsilon \sigma'_i \\ \epsilon\sigma_i' & \sigma_i\end{pmatrix}\) or \(\begin{pmatrix} \epsilon\sigma_i'\end{pmatrix}\),
  where each \(\sigma_i\) and \(\sigma_i'\) is real, and $\sigma'_i \neq 0$.
\end{theorem}

\begin{theorem}[Dual $T$-SVD]
  Given a square dual number matrix \(M \in M_n(\mathbb D)\), we can decompose the matrix as: \[M = U\Sigma V^T\] where \(U\) and \(V\) are $T$-orthogonal, and \(\Sigma\) is a diagonal matrix.
\end{theorem}

We will start by proving these theorems for invertible matrices.

\begin{theorem}[SVD for invertible matrices]
  Every invertible matrix has a \(T\)/\(*\)-SVD.
\end{theorem}
  
\begin{proof}
  We shall prove this for the \(T\)-SVD, but the argument is the same for the \(*\)-SVD.
  
  Let \(M\) be an arbitrary dual matrix. Observe that \(M^TM\) is a symmetric matrix. As such, by the \(T\)-spectral theorem, we have that \(M^TM = V\Sigma V^T\) for some orthogonal matrix \(V\) and diagonal \(\Sigma\). Also observe that the standard part of \(M^TM\) is positive-definite. From this, it follows that the standard part of \(\Sigma\) is positive. It follows that \(\sqrt{M^TM}\) exists and is equal to \(V \sqrt{\Sigma} V^T\). Observe also that \(M(\sqrt{M^T M})^{-1}\) is an orthogonal matrix which we shall call \(U\). Finally, we have that \((UV) \sqrt \Sigma V^T\) is the \(T\)-SVD of \(M\).
\end{proof}

\begin{theorem}[Dual $*$-SVD]
  Given a square dual number matrix \(M \in M_n(\mathbb D)\), we can decompose the matrix as: \[M = U\Sigma W^*\] where \(U\) and \(W\) are unitary, and \(\Sigma\) is a block-diagonal matrix where each block is either of the form \(\begin{pmatrix}\sigma_i\end{pmatrix}\), \(\begin{pmatrix}\sigma_i & -\epsilon \sigma'_i \\ \epsilon\sigma_i' & \sigma_i\end{pmatrix}\) or \(\begin{pmatrix} \epsilon\sigma_i'\end{pmatrix}\),
  and each \(\sigma_i\) and \(\sigma_i'\) is real.  
\end{theorem}
\begin{proof}
  Let $T: V \to V$ be a linear endomorphism over a finite-dimensional free $\mathbb D$-module.

  The operator $T^* T$ is Hermitian. Thus, apply the $*$-spectral theorem to it to get an orthonormal basis $B = \{b_1, b_2, \dotsc, b_n\}$. Let $V^L$ be the span of the set of vectors in $B$ that get mapped to appreciable vectors by $T$. Let $V^R$ be the span of the set of vectors in $B$ that get mapped to infinitesimal vectors by $T$.\footnote{It so happens that $V^L$ is equal to the image of $T^* T$, and that $V^R$ is equal to both $\ker(T^*T)$ and $\ker(T^tT)$ (where $T^t$ denotes the tranpose of $T$). These are easy to prove, and we won't do so here.}

  Observe that $T$ is injective over $V^L$. Let $I$ be the image of $V^L$ under $T$. Observe that $I$ has the same dimension as $V^L$. Thus, there exists a unitary operator $U$ such that $U(I) = V^L$. Clearly, $UT(V^L) = V^L$. Observe that $UT(V^R) \subseteq V^R$.

  Since $UT|_{V^L}$ is a linear endomorphism, we may take its $*$-SVD by theorem \emph{Dual $*$-SVD for invertible matrices}. We also observe that $UT|_{V^R}$ is an infinitesimal map (meaning that it maps every argument to infinitesimals), and can therefore be expressed as $\epsilon T'$. We can thus take the SVD of $T'$. Taking the direct sum of the $*$-SVDs of $UT|_{V^L}$ and $UT|_{V^R}$ yields the $*$-SVD of $UT$.

  The block types \(\begin{pmatrix}\sigma_i\end{pmatrix}\) and \(\begin{pmatrix}\sigma_i & -\epsilon \sigma'_i \\ \epsilon\sigma_i' & \sigma_i\end{pmatrix}\) come from $UT|_{V^L}$, and the block type \(\begin{pmatrix} \epsilon\sigma_i'\end{pmatrix}\) comes from $UT|_{V^R}$.

  Finally, we multiply the SVD of $UT$ on the left by $U^{-1}$, and we are done.
\end{proof}

\begin{lemma} The following claims made in the above proof are true:
  \begin{enumerate}
    \item The Gram-Schmidt process is a valid algorithm for producing an orthonormal basis for some free $\mathbb D$-module $W$ given a basis for $W$.
    \item $T$ is injective over $V^L$.
    \item There exists a unitary map $U$ such that $UT(V^L) = V^L$.
    \item $UT(V^L) \perp UT(V^R)$.
    \item $UT(V^R) \subseteq V^R$.
  \end{enumerate}
\end{lemma}
\begin{proof} We proceed to prove each claim.

  \emph{Proof of claim 1.} The proof is by induction. Assume that $\dim(W) = 0$ and the basis is $\emptyset$. Then Gram-Schmidt finishes immediately and produces the orthonormal basis $\emptyset$. Now assume that given a set of vectors $\{v_1, v_2, \dotsc, v_n, v_{n+1}\}$, the first $n$ iterations of Gram-Schmidt succeed in producing an orthonormal set of vectors $\{e_1, e_2, \dotsc, e_n\}$ whose span is equal to the span of $\{v_1, v_2, \dotsc, v_n\}$. We perform another iteration and get $e_{n+1} = \frac{v_{n+1} - \sum_{i=1}^n \langle e_i, v_{n+1} \rangle e_i}{\left|v_{n+1} - \sum_{i=1}^n \langle e_i, v_{n+1} \rangle e_i\right|}$ where $|u|$ means $|\operatorname{st}(u)|$. We need to check that the division is sound. Assuming it isn't, we would have that $\left|v_{n+1} - \sum_{i=1}^n \langle e_i, v_{n+1} \rangle e_i\right| = 0$, which implies that $\epsilon\left(v_{n+1} - \sum_{i=1}^n \langle e_i, v_{n+1} \rangle e_i\right) = 0$, which contradicts the linear independence of $\{v_1, v_2, \dotsc, v_n, v_{n+1}\}$. The linear independence, spanning property and orthonormality of $\{e_1, e_2, \dotsc, e_n, e_{n+1}\}$ are easily shown.

  \emph{Proof of claim 2.} It's easily seen that $T^* T$ is injective over $V^L$ (based on the definition of $V^L$). So then assume that $T(u) = T(v)$ for some $u$ and $v$ in $V^L$. We then have that $T^*T(u) = T^*T(v)$. By injectivity of $T^* T$, we have that $u = v$.

  \emph{Proof of claim 3.} We begin by constructing an orthonormal basis for $T(V^L)$. Clearly, given an orthonormal basis $\{b_1, b_2, \dotsc, b_k\}$ of $V^L$, we get a basis $\{T(b_1), T(b_2), \dotsc, T(b_k)\}$ of $T(V^L)$. But this latter basis may not be orthonormal. We perform Gram-Schmidt on this latter basis to get an orthonormal basis $\{c_1, c_2, \dotsc, c_k\}$ of $T(V^L)$. We extend the orthonormal basis $\{b_1, b_2, \dotsc, b_k\}$ of $V^L$ to an orthonormal basis $\{b_1, b_2, \dotsc, b_k, b_{k+1}, \dotsc, b_n\}$ of the whole space, and the orthonormal basis $\{c_1, c_2, \dotsc, c_k\}$ of $T(V^L)$ to an orthonormal basis $\{c_1, c_2, \dotsc, c_k, c_{k+1}, \dotsc, c_n\}$ of the whole space. We leave it to the reader to verify that this extension is possible. We now define $U(c_i) = b_i$ for all $i$, and observe that this $U$ is unitary, and maps $T(V^L)$ to $V^L$.

  \emph{Proof of claim 4.} We take the inner product of the spaces $UT(V^L)$ and $UT(V^R)$. We get $\langle UT(V^L), UT(V^R) \rangle = \langle T(V^L), T(V^R) \rangle = \langle T^* T(V^L), V^R \rangle = \langle V^L, V^R \rangle = 0$. We've used the fact that $T^* T(V^L) = V^L$, which should be easy to verify.

  \emph{Proof of claim 5.} Let $\{b_1, b_2, \dotsc, b_k\}$ be basis of $V^L$, and $\{b_{k+1}, b_{k+2}, \dotsc, b_n\}$ be a basis of $V^R$. Their union is a basis for the whole space. Let $v \in UT(V^R)$. Express $v$ as $\sum_{i=1}^n \lambda_i b_i$. Since $UT(V^R) \perp V^L$, we have that for $i \leq k$, $\lambda_i = 0$. It follows that $v$ is in the span of $\{b_1, b_2, \dotsc, b_k\}$, and therefore that $v$ is in $V^R$.
\end{proof}

\begin{theorem}[Dual $T$-SVD]
  Given a square dual number matrix \(M \in M_n(\mathbb D)\), we can decompose the matrix as: \[M = U\Sigma V^T\] where \(U\) and \(V\) are $T$-orthogonal, and \(\Sigma\) is a diagonal matrix.
\end{theorem}  
\begin{proof}
  Essentially the same as for the $*$-SVD, except that the role of the $*$-spectral theorem is changed to the $T$-spectral theorem.
\end{proof}

\section{Acknowledgements}
I am thankful to Gregory Gutin and Ilya Spitkovsky for helpful comments and suggestions.
The software Sympy v1.6 (\cite{10.7717/peerj-cs.103}) was helpful for finding some of the proofs.

\bibliographystyle{plain}

\end{document}